\documentclass[a4paper,10pt]{article}
\usepackage{amssymb,amsmath,amsthm}
\usepackage{fullpage}
\usepackage{hyperref}
\usepackage{eucal}
\usepackage{color}
\usepackage{stackrel}
\usepackage{graphicx}
\newcommand{\ie}{\emph{i.e.}}
\newcommand{\eg}{\emph{e.g.}}
\newcommand{\cf}{\emph{cf.}}

\newcommand{\Real}{\mathbb{R}}
\newcommand{\Com}{\mathbb{C}}
\newcommand{\Nat}{\mathbb{N}}

\newcommand{\Dom}{\mathsf{D}}
\newcommand{\Ran}{\mathsf{R}}
\newcommand{\Ker}{\mathsf{N}}

\newcommand{\obal}{\mathop{\mathrm{span}}\nolimits}
\newcommand{\cof}{\mathop{\mathrm{cof}}\nolimits}
\newcommand{\defic}{\mathsf{def}}
\newcommand{\codim}{\mathop{\mathrm{codim}}\nolimits}
\newcommand{\nul}{\mathsf{nul}}
\newcommand{\mind}{\mathsf{m}}
\newcommand{\Disk}{\mathbb{D}}
\newcommand{\Circle}{\mathbb{T}}
\newcommand{\eps}{\varepsilon}
\newcommand{\sii}{L^2}

\newcommand{\Hilbert}{\mathcal{H}}

\newtheorem{Theorem}{Theorem}
\newtheorem{Lemma}{Lemma}
\newtheorem{Proposition}{Proposition}
\newtheorem{Corollary}{Corollary}

\theoremstyle{definition}
\newtheorem{Remark}{Remark}

\newtheorem{Definition}{Definition}
\def\OMIT#1{}
\usepackage[normalem]{ulem}
\definecolor{DarkGreen}{rgb}{0,0.5,0.1} 

\newcommand\soutD{\bgroup\markoverwith
{\textcolor{DarkGreen}{\rule[.5ex]{2pt}{1pt}}}\ULon}
\newcommand\soutP{\bgroup\markoverwith
{\textcolor{blue}{\rule[.5ex]{2pt}{1pt}}}\ULon}
\newcommand{\Hm}[1]{\leavevmode{\marginpar{\tiny%
$\hbox to 0mm{\hspace*{-0.5mm}$\leftarrow$\hss}%
\vcenter{\vrule depth 0.1mm height 0.1mm width \the\marginparwidth}%
\hbox to
0mm{\hss$\rightarrow$\hspace*{-0.5mm}}$\\\relax\raggedright #1}}}

\begin{document}
%
%
\title{\textbf{\LARGE Complex-self-adjointness}}
\author{M.~Cristina C\^amara\,$^a$ \ and \ David Krej\v{c}i\v{r}{\'\i}k\,$^b$}
\date{\small 
\vspace{-5ex}
\begin{quote}
\emph{
\begin{itemize}
\item[a)] 
Center for Mathematical Analysis, Geometry and Dynamical Systems, 
Mathematics Department, Instituto Superior T\'ecnico, Universidade de Lisboa, 
Av. Rovisco Pais, 1049-001, Lisbon, Portugal;
ccamara@math.tecnico.ulisboa.pt.%
\item[b)] 
Department of Mathematics, Faculty of Nuclear Sciences and 
Physical Engineering, Czech Technical University in Prague, 
Trojanova 13, 12000 Prague 2, Czechia; 
david.krejcirik@fjfi.cvut.cz.
\end{itemize}
}
\end{quote}
11 August 2022
}
\maketitle
\vspace{-5ex} 
\begin{abstract}
\noindent
We develop the concept of operators in Hilbert spaces
which are similar to their adjoints via antiunitary operators,
the latter being not necessarily involutive.  
We discuss extension theory, refined polar 
and singular-value decompositions,
and antilinear eigenfunction expansions.
The study is motivated by physical symmetries
in quantum mechanics with non-self-adjoint operators.

%
%
\end{abstract}
%

\section{Introduction}
%
A fundamental postulate of quantum mechanics says that 
physical observables are represented by self-adjoint operators
in Hilbert spaces. 
It does not mean that quantum theory is free of non-self-adjoint operators.
For instance, physical \emph{symmetries} are represented by 
unitary and antiunitary operators.
There are also quantum \emph{open systems} 
which are conveniently modelled by non-self-adjoint operators;
more generally, the latter arise in many other areas of physics 
where the time evolution is not conservative. 
What is more, recent years have brought motivations for considering
unconventional representations of observables by possibly 
non-self-adjoint operators which are however similar
to self-adjoint operators
(see~\cite{Bagarello-book} for a mathematical overview).

To cover these more general circumstances 
in the recent developments of quantum mechanics, 
the concept of linear operators~$H$ satisfying
the \emph{complex-self-adjointness} relation 
\begin{equation}\label{C-symmetry}
  H^* = C H C^{-1}
\end{equation}
with an \emph{antiunitary}~$C$ was proposed in 
\cite{BK,KKNS,KS-book}. 
Here the physical meaning of~$C$ is typically 
a time-reversal symmetry or its composition 
with the parity and charge symmetries.

If~$C$ is \emph{involutive}, \ie~$C^2=I$ (so that $C^{-1}=C$),
then $C$~is called \emph{conjugation}
and the operators satisfying~\eqref{C-symmetry}
are sometimes called \emph{$C$-self-adjoint}
(see \cite[Sec.~I.4]{Glazman} and \cite[Sec.~III.5]{Edmunds-Evans}).
This class of operators represent a well studied area of functional analysis,
notably due to the highly influential works of Garcia and Putinar
\cite{Garcia-Putinar_2006, Garcia-Putinar_2007}
(see also \cite{Prodan-Garcia-Putinar_2006}
and the excellent survey~\cite{Garcia-Prodan-Putinar_2014}).

However, it was pointed out in~\cite{KKNS} that
the involutive requirement about~$C$
is too restrictive once the spin is taken into account. 
Indeed, the time-reversal operator for fermionic 
systems modelled by the Pauli or Dirac operators 
is rather \emph{anti-involutive}, \ie~$C^2=-I$.
To cover these and other circumstances, 
the more general identity~\eqref{C-symmetry} 
was proposed in~\cite{KKNS}
(see also~\cite[Sec.~5.2.5.4]{KS-book}).
Unfortunately, despite the prevailing physical motivations,
there seem to be no systematic theory of the general
complex-self-adjoint operators and the objective of this paper
is precisely to initiate to fill in this gap.

The organisation of this paper is as follows.
In Section~\ref{Sec.pre}, we introduce the concept
of complex-symmetric and complex-self-adjoint operators
and summarise their basic properties.
The extension theory for symmetric and complex-symmetric operators 
is discussed in Section~\ref{Sec.extension}.
We manage to go beyond the involutive approach of
\cite[Sec.~III.5]{Edmunds-Evans} in several directions,
however, the question whether any $C$-symmetric operator 
admits a $C$-self-adjoint extension remains open
(unless~$C$ is involutive). 
In Section~\ref{Sec.polar}, 
we extend the results of \cite{Garcia-Putinar_2007}
about refined polar decompositions for $C$-self-adjoint operators
to the present general setting.
In Sections~\ref{Sec.singular} and~\ref{Sec.antilinear},
we try to extend the results of \cite{Garcia-Putinar_2007}
about refined singular-value decompositions
for compact $C$-self-adjoint operators
and antilinear eigenfunction expansions 
for $C$-self-adjoint operators with compact resolvents, respectively,
however, we manage to do so only under an extra hypothesis
about the simplicity of the singular values.
In Section~\ref{Sec.toy} we present
an operator in the spinorial Hilbert space $\sii(\Real;\Com^2)$
which is complex-symmetric but for which there is 
no obvious involutive conjugation. 
Further illustrations, 
including examples in model spaces,
are given in Section~\ref{Sec.examples}.

\section{Definitions and basic properties}\label{Sec.pre}
%
Let~$\Hilbert$ be a complex separable Hilbert space;
the norm and the inner product is denoted by 
$\|\cdot\|$ and $(\cdot,\cdot)$, respectively,
and our convention is that the latter is linear in the second component.
A \emph{linear} (respectively, \emph{antilinear}) operator~$A$ in~$\Hilbert$ 
is a function which sends every vector~$\psi$ 
in a (linear) subspace $\Dom(A) \subset \Hilbert$ 
called the \emph{domain} of~$A$
to a vector $A\psi \in \Hilbert$ and satisfies 
the \emph{additivity} condition 
$
  A(\phi + \psi) =  A\phi + A\psi
$
for every $\phi, \psi \in \Dom(A)$ 
together with the \emph{homogeneity} 
(respectively, \emph{antihomogeneity}) condition
$
  A(\alpha\psi) =  \alpha A\psi
$
(respectively,
$
  A(\alpha\psi) =  \bar\alpha A\psi
$	)
for every $\psi \in \Dom(A)$ and $\alpha \in \Com$. 
If $\Dom(A) = \Hilbert$, the operator~$A$ 
is said to be defined \emph{on}~$\Hilbert$. 
We denote the \emph{range} and \emph{kernel} of~$A$ by
$\Ran(A)$ and $\Ker(A)$, respectively.

An operator~$A$ on~$\Hilbert$ is said to be \emph{isometric}
if it preserves the norm: $\|A\psi\| = \|\psi\|$
for every $\psi \in \Hilbert$.
By the polarisation identity, it follows that 
$(A\phi,A\psi) = (\phi,\psi)$ 
(respectively, $(A\phi,A\psi) = (\psi,\phi)$)
for every $\phi,\psi \in \Hilbert$,
provided that~$A$ is linear (respectively, antilinear).
In either case, the adjoint of~$A$ coincides with the left inverse of~$A$.
More specifically, in distinction to some recent papers
(see, \eg, \cite{Camara-Klis-Garlicka-Ptak_2019,Ko-Lee-Lee_2022}),
we use the same notation~$A^*$ for the adjoint of 
any densely defined operator~$A$:
it satisfies $(\phi,A\psi) = (A^*\phi,\psi)$ 
(respectively, $(\phi,A\psi) = \overline{(A^*\phi,\psi)}$),
for every $\psi \in \Dom(A)$ and $\phi \in \Dom(A^*)$,
if~$A$ is linear (respectively, antilinear).  
Then the left-inverse property 
for the isometric operator~$A$
precisely means that $A^*A = I$,  
where~$I$ stands for the identity operator on~$\Hilbert$. 
If, in addition, $\Ran(A) = \Hilbert$, 
then~$A$ is called \emph{unitary} or \emph{antiunitary}
depending on whether~$A$ is linear or antilinear, respectively.
Then the adjoint of~$A$ coincides with the right inverse of~$A$ too,
\ie\ $AA^* = I$.
In summary, $A$~is bijective as an operator on~$\Hilbert$ onto~$\Hilbert$  
and $A^{-1} = A^*$  
if~$A$ is unitary or antiunitary.
Any unitary or antiunitary operator~$A$ 
is bounded and boundedly invertible,
in fact $\|A\| = 1 = \|A^{-1}\|$.

In the sequel, $H$~stands for a generic linear operator in~$\Hilbert$.
\begin{Definition}
We say that~$H$ is \emph{complex-symmetric} \emph{with respect to}~$C$ 
(or briefly $C$-\emph{symmetric})
if~$H$ is densely defined and
there exists an antiunitary operator~$C$ in~$\Hilbert$
such that the adjoint~$H^*$
is an extension of $CHC^{-1}$, 
\ie\  $CHC^{-1} \subset H^*$.

We say that~$H$ is \emph{complex-self-adjoint} 
\emph{with respect to}~$C$
(or briefly $C$-\emph{self-adjoint})
if~$H$ is densely defined and 
there exists an antiunitary operator~$C$ in~$\Hilbert$
such that $CHC^{-1} = H^*$.

We say that~$H$ is \emph{complex-symmetric} 
(respectively, \emph{complex-self-adjoint})
if there exists an antiunitary operator 
with respect to which it is complex-symmetric
(respectively, complex-self-adjoint).
\end{Definition}
\begin{Remark}\label{Rem1} 
Since we allow for the concise notations 
``$C$-symmetric'' and ``$C$-self-adjoint'',
it is important to emphasise that our definition generalises 
the usual concepts of \emph{just involutive} antiunitary~$C$ 
(see \cite[Sec.~I.4]{Glazman} and \cite[Sec.~III.5]{Edmunds-Evans}
for the traditional references);
in this special case, $C$~is called \emph{conjugation}.
Occasionally, we shall use the same terminology 
in the present more general setting
when~$C$ is not necessarily involutive.

At the same time, we warn the reader that the same notation
is used in the theory of Krein spaces 
where~$C$ is a \emph{linear involution}
(see, \eg, \cite{Azizov-Iokhvidov1_1989}).
Even more confusingly, in physics literature,
the property of $H$ being ``$C$-symmetric'' 
(with $C$ unitary or antiunitary)
occasionally means the commutation relation $[C,H]=0$
(which precisely means $CH = HC$).
We refer to~\cite[Sec.~5.2.5]{KS-book} for an overview 
of these different notions.
\end{Remark}

The $C$-self-adjointness of~$H$ 
particularly requires that
$C:\Dom(H) \to \Dom(H^*)$ is a bijection.
Obviously, $H$ is $C$-self-adjoint,
if, and only if, $H^*$ is $C^{-1}$-self-adjoint.
 
If~$C$ is antiunitary,
it is easily seen that 
\begin{equation}\label{crucial}
  (\phi,C\psi) = (\psi,C^{-1}\phi)
\end{equation}  
for every $\phi,\psi \in \Hilbert$.
Furthermore, for any densely defined linear operator~$H$,
it is easily shown that
\begin{equation}\label{5.3}
  (CHC^{-1})^* = CH^*C^{-1}
  \,.  
\end{equation}  
It follows that~$H$ is $C$-symmetric
if, and only if,
$H \subset C^{-1}H^*C = (C^{-1}HC)^*$
and this implies that~$H$ is closable,
the closure being also $C$-symmetric.
At the same time, by the closedness of the adjoint,
any complex-self-adjoint operator is automatically closed.
 	
Recall that the \emph{spectrum} $\sigma(H)$ of 
any closed linear operator~$H$ in~$\Hilbert$  
is the set of those complex numbers~$\lambda$ for which
$H-\lambda I:\Dom(H)\to\Hilbert$ is not bijective.
The \emph{resolvent set} is the complement $\rho(H):=\Com\setminus\sigma(H)$.
The \emph{point spectrum} $\sigma_{\mathrm{p}}(H)$ of~$H$ 
is the set of eigenvalues of~$H$
(\ie~the operator $H-\lambda I$ is not injective).
For the surjectivity, one says that 
$\lambda \in \sigma(H)$ belongs to the \emph{continuous spectrum}
$\sigma_{\mathrm{c}}(H)$
(respectively, \emph{residual spectrum} $\sigma_{\mathrm{r}}(H)$) of~$H$ 
if $\lambda\not\in\sigma_\mathrm{p}(H)$
and the closure of the range $\Ran(H-\lambda I)$ equals~$\Hilbert$  
(respectively, the closure is a proper subset of~$\Hilbert$).
If~$H$ is densely defined, then $\Ker(H^*) = \Ran(H)^\bot$,
where~$\bot$ denotes the orthogonal complement; 	
consequently, one has the general characterisation
\begin{equation}\label{residual}
  \sigma_{\mathrm{r}}(H) 
  = \left\{\lambda \not\in \sigma_{\mathrm{p}}(H): \
  \bar{\lambda} \in \sigma_{\mathrm{p}}(H^*)\right\}
  \,.
\end{equation}

It is well known that the residual spectrum 
of any self-adjoint operator is empty.
It turns out that the same holds true 
for complex-self-adjoint operators.
In the context of applications in quantum mechanics,
this simple observation goes back to~\cite{BK} and~\cite{KKNS}
in the involutive and the general case, respectively. 	
\begin{Proposition}\label{Prop.residual} 
Let~$H$ be a linear operator 
which is complex-self-adjoint. Then 
$$
  \sigma_{\mathrm{r}}(H) = \varnothing \,.
$$
\end{Proposition}
\begin{proof}
Let~$H$ satisfy~\eqref{C-symmetry} with some antiunitary operator~$C$.
Then~$\lambda$ is an eigenvalue of~$H$ 
(with eigenvector $\psi \in \Dom(H)$)
if, and only if,  	
$\bar\lambda$ is an eigenvalue of~$H^*$
(with eigenvector $C\psi \in \Dom(H^*)$).
So the absence of elements in the residual spectrum follows 
from the general formula~\eqref{residual}.
\end{proof} 	
%

\section{Extension theory}\label{Sec.extension}
%
The \emph{nullity} (respectively, \emph{deficiency}) 
of a linear operator~$H$ in~$\Hilbert$
is $\nul(H) := \dim\Ker(H)$
(respectively, $\defic(H) := \codim\Ran(H)$).
Recall that the codimension of a subspace $\Hilbert' \subset \Hilbert$
is defined as the dimension of the quotient space $\Hilbert/\Hilbert'$.
If~$\Hilbert'$ is closed, then $\codim\Hilbert' = \dim\Hilbert'^\bot$,
where~$\bot$ denotes the orthogonal complement.

If~$H$ is closed, it is well known 
(see, \eg, \cite[Thm.~III.2.3]{Edmunds-Evans})
that $\defic(H - \lambda I)$
is constant in each connected component of the exterior 
of the numerical range of~$H$. 
In particular, if~$H$ is symmetric
(\ie\ $H$~is densely defined and $H \subset H^*$), 
then the \emph{deficiency indices}:
$$
  \mind_\pm(H) := \defic(H - \lambda I)
  \,, \qquad
  \lambda \in \Com_\mp := \{\lambda \in \Com: 
  \Im\lambda \lessgtr 0 \}
$$
are well defined and constant. 
By extension theory of symmetric operators 
(see, \eg, \cite[Sec.~III.4]{Edmunds-Evans}),
a symmetric operator $H$~admits a self-adjoint extension
if, and only if, $\mind_+(H) = \mind_-(H)$;
$H$~is self-adjoint if, and only if, $\mind_+(H) = 0 = \mind_-(H)$.

There are well-known examples of symmetric operators
which do not admit any self-adjoint extension.
It is interesting that these pathological situations
can be excluded for operators admitting an antiunitary symmetry
in the sense of the following definition.

\begin{Definition}
We say that~$H$ is \emph{real} \emph{with respect to}~$C$ 
(or briefly $C$-\emph{real})
if there exists an antiunitary operator~$C$ in~$\Hilbert$
such that $CHC^{-1} = H$.
\end{Definition}
\begin{Remark} 
The definition is an extension of the terminology of
\cite[Thm.~III.5.1]{Edmunds-Evans} to not necessarily
involutive operators~$C$.
In physical terms, the present general notion is equivalent to
the commutation relation $[C,H]=0$ (\cf~Remark~\ref{Rem1}).
\end{Remark}

The following theorem 
is an extension of \cite[Thm.~III.5.3]{Edmunds-Evans},
where the symmetry~$C$ was assumed to be involutive. 

\begin{Theorem} 
Let~$H$ be a closed symmetric operator which is $C$-real.
Then $\mind_+(H) = \mind_-(H)$.
Consequently, $H$~admits a self-adjoint extension.
\end{Theorem}
\begin{proof}
Since~$H$ is $C$-real, it follows from~\eqref{5.3} 
that also the adjoint~$H^*$ is $C$-real.
Note also that the range $\Ran(H-\lambda I)$ is closed
for $\lambda \in \Com_\pm$.
Consequently,
$$
\begin{aligned}
  \Hilbert / \Ran(H + i I)
  \cong \Ran(H + i I)^\bot 
  &= \Ker(H^* - i I) \\
  &= \Ker((CHC^{-1})^*-iI) \\
  &= \Ker(CH^*C^{-1}-iI) \\
  &= \Ker(C(H^*+iI)C^{-1}) \\
  &= C \Ker(H^*+iI)
  \,.
\end{aligned}  
$$
Since $\mind_\pm(H) = \dim\Ker(H^* \mp i I)$
and~$C$ is a bijection,
we have $\mind_+(H) = \mind_-(H)$.
\end{proof}

We now turn to complex-symmetric operators. 	 
The following simple observation
is an extension of \cite[Lem.~III.5.4]{Edmunds-Evans},
where the operator is assumed to be 
complex-symmetric/complex-self-adjoint 
with respect to (involutive) conjugations. 

\begin{Lemma}\label{Lem.5.4} 
Let~$H$ be a closed complex-symmetric operator.
Then, for any $\lambda \in \Com$,
\begin{equation}\label{inclusion}
  \Ker(H-\lambda I) \subset C^{-1}\Ker(H^*-\bar\lambda I) \,,
\end{equation}
so
$$
  \nul(H-\lambda I) \leq \nul(H^*-\bar\lambda I)
  \,.
$$
If~$H$ is complex-self-adjoint, 
the inclusion and the inequality become equalities.
\end{Lemma}
\begin{proof} 
If~$H$ is $C$-symmetric, 
then $H \subset C^{-1} H^* C$, so 
for $f \in \Dom(H)$
$$
\begin{aligned}
  (H-\lambda I)f = 0
  &\Longleftrightarrow  
  (C^{-1}H^*C-\lambda I)f = 0 \\
  &\Longleftrightarrow  
  C^{-1}(H^*-\bar\lambda I)C f =0 \\
  &\Longleftrightarrow 
  (H^*-\bar\lambda I)C f =0 \\
  &\Longleftrightarrow 
  f \in C^{-1} \Ker(H^*-\bar\lambda I)
  \,.
\end{aligned}  
$$ 
Thus, \eqref{inclusion}~holds.  
If~$H$ is $C$-self-adjoint, 
then $H = C^{-1} H^* C$, 
so $f \in \Dom(H)$ if, and only if, $Cf \in \Dom(H^*)$,
therefore we have equality in~\eqref{inclusion}.
\end{proof}

Before we proceed further, 
let us recall 
(see, \eg, \cite[Def.~III.2.5]{Edmunds-Evans})
that the \emph{field of regularity} $\Pi(H)$ of~$H$ 	 
is defined as the set of complex points~$\lambda$
for which there exist positive constants $k_\lambda$ 
such that $\|(H-\lambda I) \psi\| \geq k_\lambda \|\psi\|$
for all $\psi \in \Dom(H)$. 	
If~$H$ is closed, then $\Pi(H)$ coincides 
with the set of points~$\lambda$ for which 
$\Ran(H-\lambda I)$ is closed and $\nul(H-\lambda I)=0$.
In this case, $\lambda \mapsto \defic(H-\lambda I)$ 
is constant in any connected component of $\Pi(H)$. 
 	 
Now we are in a position to state the following generalisation
of \cite[Thm.~III.5.5]{Edmunds-Evans}.

\begin{Theorem} 
A closed $C$-symmetric operator is $C$-self-adjoint
if, and only if, $\defic(H-\lambda I)=0$ for some,
and hence all, $\lambda \in \Pi(H)$.  
\end{Theorem}
\begin{proof} 
If~$H$ is complex-self-adjoint and $\lambda \in \Pi(H)$,
then $\Ran(H-\lambda I)$ is closed
and Lemma~\ref{Lem.5.4} implies
$$
  \defic(H-\lambda I) 
  = \nul(H^*-\bar\lambda I)
  = \nul(H-\lambda I)
  =0 \,.
$$ 

Conversely, suppose that $\defic(H-\lambda I)=0$
for some $\lambda \in \Pi(H)$.
Then $\Ran(H-\lambda I)=\Hilbert$ and 
\begin{equation}\label{5.5}
  \Ran(CHC^{-1}-\bar\lambda I)
  = \Ran(C(H-\lambda I)C^{-1})
  = \Hilbert 
  \,.
\end{equation}
If $CHC^{-1} \not= H^*$, there exists a non-zero
$\psi \in \Dom(H^*) \setminus \Dom(CHC^{-1})$.  
From~\eqref{5.5}, 
$$
  (H^* - \bar\lambda I)\psi 
  =  (CHC^{-1}-\bar\lambda I) \phi 
$$
for some $\phi \in \Dom(CHC^{-1})$
and since $CHC^{-1} \subset H^*$ we have 
$(H^* - \bar\lambda I) (\psi-\phi) = 0$.
That is, $\psi-\phi \in \Ker(H^* - \bar\lambda I)$.
But 
$
  \Ker(H^* - \bar\lambda I)
  = \Ran(H - \lambda I)^\bot
  = \{0\}
$
and consequently $\psi = \phi \in \Dom(CHC^{-1})$,
contradicting the hypothesis that $CHC^{-1} \not= H^*$. 
\end{proof}

Now, following \cite[Sec.~III.5.2]{Edmunds-Evans},
given a $C$-symmetric operator~$H$,
we define $A:=H$ and $B:=C^{-1} H C$.
Then, for every $\psi \in \Dom(A) = \Dom(H)$
and $\phi \in \Dom(B) = C^{-1}\Dom(H)$, 
one has 
$$
  (\phi,A\psi)
  = (\phi,C^{-1}H^*C\psi)
  = (H^*C\psi,C\phi)
  = (C\psi,HC\phi)
  = (B\phi,\psi)
  \,,
$$ 	 
where the second and last equalities employ~\eqref{crucial}. 	 
Hence, $A$ and~$B$ are adjoint to each other.  	 
Also if $\lambda \in \Pi(A)$ and $\psi \in \Dom(B)$,
$$
  \|(B-\bar\lambda)\psi\|
  = \|C^{-1}(H-\lambda I)C\psi\|
  \geq k_\lambda \;\! \|C\psi\|
  = k_\lambda \;\! \|\psi\|
  \,,
$$ 	 
so that $\bar\lambda \in \Pi(B)$.
Then the theory of extensions of closed operators
(see \cite[Sec.~III.3]{Edmunds-Evans})
yields the general formulae
$$
\begin{aligned}
  \Dom(B^*) &= \Dom(A) \dot{+} \Ker((A^*-\bar\lambda I)(B^*-\lambda I)) 
  \,,
  \\ 
  \Dom(A^*) &= \Dom(B) \dot{+} \Ker((B^*-\lambda I)(A^*-\bar\lambda I)) 
  \,.
\end{aligned}  
$$
Moreover, if $\defic(A-\lambda I)$ and $\defic(B-\bar\lambda I)$
are finite, then
$$
  \dim(\Dom(B^*)/\Dom(A))
  = \dim(\Dom(A^*)/\Dom(B))
  = \defic(A-\lambda I) + \defic(B-\bar\lambda I)\
  \,.
$$
In our case, if $\lambda \in \Pi(H)$, then
$$
\begin{aligned}
  \defic(A-\lambda I)
  &= \nul(H^*-\bar\lambda I)
  \\
  &= \nul(C^{-1}(H^*-\bar\lambda I)C)
  \\
  &= \nul(C^{-1}H^*C-\lambda I)
  \\
  &= \defic((C^{-1}H^*C)^*-\bar\lambda I)
  \\
  &= \defic(B-\bar\lambda I)
  \,.
\end{aligned}  
$$ 	 
We therefore obtain the following generalisation of 	 
\cite[Thm.~III.5.6]{Edmunds-Evans}.
\begin{Theorem} 
Let~$H$ be a closed $C$-symmetric operator 
with $\Pi(H) \not= \varnothing$. 
Then, for any $\lambda \in \Pi(H)$,
$$
  \Dom(C^{-1}H^* C) 
  = \Dom(H) \dot{+} \Ker((H^*-\bar\lambda I)(C^{-1}H^*C-\lambda I)) 
  \,.
$$   
If $\defic(H-\lambda I) < \infty$,
$$
  \dim(\Dom(C^{-1}H^* C)/\Dom(H))
  = 2 \defic(H-\lambda I)
  \,.
$$
\end{Theorem}

An immediate consequence of this theorem is the following 
generalisation of \cite[Thm.~III.5.7]{Edmunds-Evans}.

\begin{Corollary} 
If~$H$ is a closed complex-symmetric operator, 
then $\lambda \mapsto \defic(H-\lambda I)$
is constant on $\Pi(H)$.
\end{Corollary}

We leave as an open problem whether 
any $C$-symmetric operator has a $C$-self-adjoint extension.
This is well known to hold if~$C$ is involutive
(see \cite[Thm.~III.5.8]{Edmunds-Evans}).
However, the proof of \cite[Thm.~III.5.8]{Edmunds-Evans}
does not seem to extend to the general situation
of the present paper.

\section{Refined polar decomposition}\label{Sec.polar}
%
Following \cite[Sec.~V.2.2]{Kato},
we first extend the notion of isometric operators introduced above.
We say that a (linear or antilinear) 
operator~$A$ defined on~$\Hilbert$
is \emph{partially isometric} if there exists 
a closed subspace $\mathcal{M} \subset \Hilbert$
such that $\|A\psi\|=\|\psi\|$ for $\psi \in \mathcal{M}$
while $A\psi = 0$ for $\psi \in \mathcal{M}^\bot$.
The closed subspaces $\mathcal{M}$ and $\mathcal{M}':=A\mathcal{M}$
are called the \emph{initial} and \emph{final} sets of~$A$,
respectively.
Then we say that a partially isometric operator~$A$ 
is \emph{partially unitary} or \emph{partially antiunitary}
depending on whether~$A$ is linear or antilinear, respectively. 

It is well known (see, \eg, \cite[Sec.~VI.2.7]{Kato})
that any densely defined, closed linear operator~$H$ in~$\Hilbert$
admits a unique \emph{polar decomposition} $H = U|H|$,
where $|H|:=(H^*H)^{1/2}$ is non-negative
and~$U$ is partially unitary  
with the initial set $\overline{\Ran(|H|)}$ 
and the final set $\overline{\Ran(H)}$.
One has $\Dom(|H|) = \Dom(H)$ and $\Ker(|H|) = \Ker(H)$.
At the same time, the polar decomposition 
of the adjoint reads $H^* = |H| U^*$.

The following refinement shows that if~$H$ is complex-self-adjoint,
then~$U$ is also complex-self-adjoint 
(with respect to the same antiunitary operator~$C$).
Furthermore, $U$ is a composition of 
another partially antiunitary operator commuting with~$H$
and the inverse of the original antiunitary operator.
In the case of involutive antiunitary operators~$C$,
the result is due to \cite[Thm.~2 \& Thm.~9]{Garcia-Putinar_2007}
and \cite[Thm.~9]{Garcia-Putinar_2007}.
If~$H$ is unbounded, however, it is assumed as an extra 
hypothesis in \cite[Thm.~9]{Garcia-Putinar_2007}
that zero is in the resolvent set of~$H$.
In our version of the theorem below,
$C$~does not need to be involutive and~$H$ can be unbounded
without any further properties. 

\begin{Theorem}\label{Thm.polar} 
Let~$H$ be a linear operator which is $C$-self-adjoint.
Then 
\begin{equation}\label{polar}
  H = C^{-1}J|H|  
  \,,
\end{equation}
where~$J$ is a partially antiunitary operator
with the initial set $\overline{\Ran(|H|)}$ 
and the final set $\overline{\Ran(|H|)}$,
which commutes with~$|H|$.
\end{Theorem}

Before proving the theorem, 
let us argue that the result indeed implies 
that the partially unitary map from the polar 
decomposition of~$H$ is complex-self-adjoint.
\begin{Corollary}
Let~$H$ be a linear operator which is $C$-self-adjoint.
If $H = U |H|$ is its polar decomposition,
then~$U$ is $C$-self-adjoint, too.
\end{Corollary}
\begin{proof}
By Theorem~\ref{Thm.polar}, one has $U=C^{-1} J$.
The complex-self-adjointness of~$H$,
identity~\eqref{polar} 
and the commutativity of~$|H|$ and~$J$
imply
$
  H^* = CHC^{-1} = J|H|C^{-1} = |H|JC^{-1} = |H| U^*
$. 	
From this and the uniqueness of the polar decomposition of~$H^*$,
we deduce that 
$
  U^* = J C^{-1} = C U C^{-1}
$,
which is the desired claim. 	
\end{proof}

The ``hidden symmetry'' $J$ shares the same involutive 
properties as the ``obvious symmetry'' $C$.
\begin{Corollary}
Assume the hypotheses and notations of Theorem~\ref{Thm.polar}.
If~$C$ is involutive (respectively, anti-involutive),
then the restriction of~$J$ on $\overline{\Ran(|H|)}$
is involutive (respectively, anti-involutive).
\end{Corollary}
\begin{proof}
As in the proof of the previous corollary,
one has $U^* = JC^{-1}$ where $U:=C^{-1}J$. 
Consequently, for every $\psi \in \overline{\Ran(|H|)}$,
one has
$
  J^2\psi 
  = U^*CCU \psi 
  = \pm U^*U \psi 
  = \pm \psi 
$,
where the plus (respectively, minus) sign holds
if~$C$ is involutive (respectively, anti-involutive).	
\end{proof}

Now we proceed with the proof of Theorem~\ref{Thm.polar}.

\begin{proof}[Proof of Theorem~\ref{Thm.polar}]
We extend the proof of \cite[Thm.~2]{Garcia-Putinar_2007} 
to the case of~$C$ being not necessarily involutive. 
Write the polar decomposition $H = U|H|$ 
and note that the complex-self-adjointness of~$H$ implies
\begin{equation}\label{one}
  H 
  = C^{-1} H^* C 
  = C^{-1} |H| U^* C
  = C^{-1} (U^*U) |H| U^* C  
  = \underbrace{(C^{-1} U^* C)}_{=:W}
  \underbrace{(C^{-1} U |H| U^* C)}_{=:A}
  \,.
\end{equation}
Here the last but one equality employs that $U^*U\psi = \psi$ 
for every $\psi \in \overline{\Ran(|H|)}$,
the initial set of~$U$.

First, we observe that~$W$ is partially isometric.
Indeed, using that $W^* = C^{-1}UC$, it is easily seen that
$
  WW^*W 
  = C^{-1} U^*UU^* C
  = W
$,
so the desired property holds due to
the criterion \cite[Prob.~V.2.6]{Kato}.
Second, we claim that~$A$ is non-negative.
Indeed, using~\eqref{crucial}, one has
$
  (\psi,A\psi) 
  = (U|H|U^*C\psi,C\psi)
  = (U^*C\psi,|H|U^*C\psi)
$
for every $\psi \in \Dom(H)$,
so the desired property follows by the non-negativity of~$|H|$.
Then the strategy of the proof is to show that the initial
set of~$W$ coincides with the initial set of~$U$.
By the uniqueness of the polar decomposition of~$H$,
it will allow us to conclude that 
\begin{equation}\label{unique}
  W=U \qquad \mbox{and} \qquad A=|H|
  \,.
\end{equation}

To show that the initial set of~$W$  
coincides with the initial set of~$U$,
it is enough to prove that 
$
  \Ker(W) = \Ker(U) = \Ker(H)
$.
Obviously, $\Ker(W) = \Ker(U^*C)$.
At the same time,
$
  \Ker(A) 
  = \Ker(U |H| U^* C) 
  = \Ker(|H| U^* C) 
  = \Ker(U^* C) 
  \,,
$
where the last two equalities hold due to the facts 
that~$U$ and~$U^*$ have $\overline{\Ran(H)}$
as their initial and final sets, respectively,
and the last equality also employs that~$|H|$ is self-adjoint.
Hence, $\Ker(W) = \Ker(A)$.
It remains to show that $\Ker(A)=\Ker(H)$.
The inclusion $\Ker(A) \subset \Ker(H)$ is clear from~\eqref{one}.
Conversely, if $\psi \in \Ker(H)$,
then~\eqref{one} implies that $\psi \in \Ker(A)$  
or $A\psi \in \Ker(W)$.  
But the latter is impossible unless $A\psi = 0$, 
because of the previously established fact $\Ker(W) = \Ker(A)$ 
and the self-adjointness of~$A$.
In summary, we have proved~\eqref{unique}.

The first equality of~\eqref{unique} shows 
that~$U$ is $C$-self-adjoint. 
Moreover, defining $J:=CU$, one has $J=U^*C$.
It follows that~$J$ is partially antiunitary
with $\overline{\Ran(|H|)}$ being both the initial and the final set.
As a consequence of the second equality of~\eqref{unique}
and the formula $J^{-1} = C^{-1}U$, one has $|H| = J^{-1}|H|J$.
Hence, $J|H| = |H|J$, so~$J$ commutes with~$|H|$.
\end{proof} 	
%
 	
\section{Refined singular-value decomposition}\label{Sec.singular}
%
We now turn to compact operators.  	
Before stating and proving the refined singular-value decomposition,
we establish an elementary result.

\begin{Lemma}\label{Lem.problem}
Let $m \in \Nat^* := \Nat\setminus\{0\}$.
Let~$\mathcal{E}$ be an $m$-dimensional subspace of~$\Hilbert$,
let~$J$ be an antiunitary operator on~$\mathcal{E}$
and assume that~$\mathcal{E}$ is invariant under~$J$. 
There exists an orthonormal basis $(\phi_1,\dots,\phi_m)$
of~$\mathcal{E}$ which is \emph{fixed} by~$J$,
\ie\ $J\phi_j=\phi_j$ for every $j \in \{1,\dots,m\}$,
provided one of the following conditions hold:
\begin{enumerate}
\item[\emph{(i)}]
$J$ is involutive;
\item[\emph{(ii)}]
$m=1$.
\end{enumerate}
\end{Lemma}
\begin{proof}
The situation~(i) is well known 
(see, \eg, \cite[Lem.~1]{Garcia-Putinar_2006}), so we only prove~(ii). 
Let $\psi \in \mathcal{E}$ be such that $\|\psi\|=1$. 
Since $\mathcal{E}=\obal(\psi)$ in invariant under isometric~$J$,
there exists $\alpha \in \Real$ such that $J\psi = e^{i\alpha}\psi$.
Let us define $\phi := e^{i\beta}\psi$,
where $\beta \in \Real$ is to be chosen by requiring $J\phi=\phi$.
An obvious solution is given by $\beta:=\alpha/2$.
\end{proof} 	
\begin{Remark}
The conclusion of Lemma~\ref{Lem.problem} 
does not hold for a general non-involutive antiunitary~$J$ 
unless the extra hypothesis~(ii) is assumed.
Indeed, since~$\mathcal{E}$ is invariant under the isometry~$J$, 
it is easily seen that, 
if $(\psi_1,\dots,\psi_m)$ is an orthonormal basis of~$\mathcal{E}$,
then $(J\psi_1,\dots,J\psi_m)$ is one, too.
Hence, there exist complex numbers $a_{jk}$
with $j,k \in \{1,\dots,m\}$ such that 
$J\psi_k = \sum_{j=1}^m a_{jk}\psi_j$
and the matrix $A:=(a_{jk})$ is unitary.
In fact, $a_{jk}=(\psi_j,J\psi_k)$.
Let $B:=(b_{jk})$ be another unitary matrix
and set $\phi_{k}:=\sum_{j=1}^m b_{jk}\psi_j$.
Then the requirement $J\phi_j = \phi_j$
for every $j \in \{1,\dots,m\}$ is equivalent 
to the matrix identity 
\begin{equation}\label{matrix}
  B^* A \bar{B} = I
  \,,
\end{equation}
where~$I$ denotes the identity matrix in~$\Com^{m \times m}$.
Using the unitarity of~$B$, \eqref{matrix}~is equivalent to $A \bar{B} = B$.
Since $\bar{B}^T=B^*=B^{-1}=(\det B)^{-1} (\cof B)^T$,
where $\cof B$ denotes the cofactor matrix of~$B$,
it follows that \eqref{matrix}~is equivalent to
\begin{equation}\label{matrix.cof}
  A \cof B = (\det B) B
  \,.
\end{equation}
If $m=2$, then~\eqref{matrix.cof} can be written as 
the homogeneous system  
$$
  \begin{pmatrix}
    -\det B & -a_{12} & 0 & a_{11} \\
    a_{12} & -\det B & -a_{11} & 0 \\
    0 & -a_{22} & -\det B  & a_{21} \\
    a_{22} & 0 & -a_{21} & -\det B
  \end{pmatrix}
  \begin{pmatrix}
    b_{11} \\ b_{12} \\ b_{21} \\ b_{22}
  \end{pmatrix}
  =
  \begin{pmatrix}
    0 \\ 0 \\ 0 \\ 0
  \end{pmatrix}
  \,.
$$
A necessary condition to guarantee the existence
of a non-trivial solution is that the determinant 
of the square matrix equals zero:
$$
  (\det A)^2 + (\det B)^4 - (\det B)^2 \, (2 a_{11}a_{22}-a_{12}^2-a_{21}^2) = 0
  \,.
$$
Using that~\eqref{matrix} implies $(\det B)^2 = \det A$,
the last condition reads
$$
  \det A \, (a_{12} - a_{21})^2 = 0
  \,,
$$
whence $a_{12} = a_{21}$.
Using the antiunitarity of~$J$ (\cf~\eqref{crucial}),
this symmetry requirement is equivalent to
$(\psi_1,J\psi_2) = (\psi_1,J^{-1}\psi_2)$,
which cannot be guaranteed in general unless $J^{-1}=J$.
\end{Remark}

Let~$H$ be a compact linear operator on~$\Hilbert$.
Recall that the eigenvalues of the self-adjoint operator~$|H|$ 
are called the \emph{singular values} of~$H$.
They will be denoted by $\sigma_1, \sigma_2, \dots$,
arranged so that $\sigma_1 \geq \sigma_2 \geq \dots \geq 0$,
and repeated according to multiplicity.
With the convention that~$\sigma_j$ is defined  
to be zero for all sufficiently large~$j$
if~$H$ (and hence~$|H|$) is of finite rank,
we see that in all cases, $\mu_j \to 0$ as $j \to \infty$.
It is well known that~$H$ and~$H^*$ 
have the same singular values
(\cf~\cite[Thm.~II.5.7]{Edmunds-Evans}).
 	
In the case of involutive antiunitary operators,
the following result is due to \cite[Thm.~3]{Garcia-Putinar_2007}. 	
\begin{Theorem}\label{Thm.compact} 	
Let~$H$ be a compact linear operator on~$\Hilbert$
which is $C$-self-adjoint.	
Assume that all the non-zero singular values~$\sigma_j$ of~$H$ 
have multiplicity one or that~$C$ is involutive.
There exist orthonormal eigenvectors~$\phi_j$ of~$|H|$ 
corresponding to the non-zero eigenvalues $\sigma_j$ 
such that
\begin{equation}\label{compact}
  H = \sum_{j=1}^\infty \sigma_j \, C^{-1}\phi_j \, (\phi_j,\cdot)
  \qquad \mbox{and} \qquad
  H^* = \sum_{j=1}^\infty \sigma_j \, \phi_j \, (C^{-1}\phi_j,\cdot)
  \,.
\end{equation}
\end{Theorem} 	
\begin{proof}	
It is well known (see, \eg, \cite[Thm.~II.5.7]{Edmunds-Evans}) 	
that any compact linear operator~$H$ on~$\Hilbert$
and its adjoint~$H^*$ admit the decompositions
\begin{equation}\label{known}
  H = \sum_{j=1}^\infty \sigma_j \, \xi_j \, (\psi_j,\cdot)
  \qquad \mbox{and} \qquad
  H^* = \sum_{j=1}^\infty \sigma_j \, \psi_j \, (\xi_j,\cdot)
  \,,
\end{equation}
where~$\psi_j$ are orthonormal eigenvectors of~$|H|$
corresponding to the eigenvalues~$\sigma_j$
and $\xi_j := \sigma_j^{-1} H\psi_j$ ($\sigma_j \not= 0$).  
The series in~\eqref{known} are finite if~$H$ is of finite rank. 
Using Theorem~\ref{Thm.polar}, we get 
$$
  \xi_j = C^{-1} J\psi_j \,.
$$
Moreover, since~$J$ commutes with~$|H|$,
it follows that $J\psi_j$ is an eigenvector of~$|H|$, too. 

Since $H$ is compact, 
the mutually orthogonal eigenspaces $\mathcal{E}_n$ of~$|H|$ 
corresponding to the \emph{distinct} non-zero eigenvalues $\mu_n$ 
are finite dimensional, say of dimension~$m_n$. 
Relabelling the terms in the known decompositions~\eqref{known},
we may write	
$$
  H = \sum_{n=1}^\infty 
  \mu_n
  \sum_{k=1}^{m_n}
  C^{-1} J \psi_{n,k} \, (\psi_{n,k},\cdot)
  \qquad \mbox{and} \qquad
  H^* = \sum_{n=1}^\infty 
  \mu_n
  \sum_{k=1}^{m_n}
  \psi_{n,k} \, (C^{-1} J\psi_{n,k},\cdot)
  \,,
$$
where $(\psi_{n,1},\dots,\psi_{n,m_n})$
is an orthonormal basis of~$\mathcal{E}_n$.
On each spectral subspace~$\mathcal{E}_n$ of~$|H|$,  
$J$~restricts to an antiunitary operator 
and $J\mathcal{E}_n = \mathcal{E}_n$.
By Lemma~\ref{Lem.problem},
there exists another orthonormal basis 
$(\phi_{n,1},\dots,\phi_{n,m_n})$ of~$\mathcal{E}_n$
which is fixed by~$J$.
Consequently,
$$
  H = \sum_{n=1}^\infty 
  \mu_n
  \sum_{k=1}^{m_n}
  C^{-1} \phi_{n,k} \, (\phi_{n,k},\cdot)
  \qquad \mbox{and} \qquad
  H^* = \sum_{n=1}^\infty 
  \mu_n
  \sum_{k=1}^{m_n}
  \phi_{n,k} \, (C^{-1}\phi_{n,k},\cdot)
  \,.
$$
The desired representation~\eqref{compact} 
follows upon a suitable relabelling of terms.
\end{proof} 	
\begin{Remark}
If~$H$ is $C$-self-adjoint,
then~$H^*$ is $C^{-1}$-self-adjoint.
Therefore one also has the alternative decomposition 
\begin{equation}\label{compact.adjoint} 
  H^* = \sum_{j=1}^\infty \sigma_j \, C \eta_j \, (\eta_j,\cdot)
  \,,
\end{equation}
where~$\eta_j$ are orthonormal eigenvectors of~$|H^*|$
corresponding to the non-zero eigenvalues~$\sigma_j$
(recall that~$H$ and~$H^*$ have the same singular values).  
The second decomposition of~\eqref{compact}
is consistent with~\eqref{compact.adjoint}.  
Indeed, from the formula $|H^*| = U |H| U^*$ 
(\cf~\cite[Eq.~(VI.2.25)]{Kato}),
where~$U$ is the partially unitary operator 
from the polar decomposition of~$H$,  
it is clear that $\phi_j$ is an eigenvector of~$|H|$ 
if, and only if, 
$
  \eta_j := U\phi_j  
$ 
is an eigenvector of~$|H^*|$. 
However, in our case,
$
  \eta_j = C^{-1} J \phi_j = C^{-1} \phi_j
$,
where the last equality holds 
because~$\phi_j$ has been fixed by~$J$  
in the proof of Theorem~\ref{Thm.compact}.
\end{Remark}
%

\section{Antilinear eigenfunction expansions}\label{Sec.antilinear}
%
There is no spectral-type theorem for complex-self-adjoint operators in general.
In fact, there are well known examples of complex-self-adjoint operators
with compact resolvent and empty spectrum (see, \eg, \cite[Sec.~VII.A]{KSTV}).
Nevertheless, even for such pathological operators, 
there is always a canonically associated antilinear eigenvalue problem 
which has a complete set of mutually orthogonal eigenfunctions.
This is the message of the following result,
which is due to \cite[Thm.~7]{Garcia-Putinar_2007}
in the case of involutive antiunitary operators.
\begin{Theorem}\label{Thm.expansion} 	
Let~$H$ be a $C$-self-adjoint
linear operator in an infinite-dimensional~$\Hilbert$.
Assume that~$H$ has a compact resolvent $R(z):=(H - z I)^{-1}$
for some complex number~$z$.
Furthermore, assume that all the singular values of~$R(z)$
have multiplicity one or that~$C$ is involutive.
There exists an orthonormal basis $(\psi_j)_{j=1}^\infty$ of~$\Hilbert$
consisting of solutions of the antilinear eigenvalue problem
\begin{equation}\label{expansion}
  (H - z I) \psi_j = \lambda_j C\psi_j
  \,,
\end{equation}
where $(\lambda_j)_{j=1}^\infty$ is an increasing 
sequence of positive numbers tending to~$+\infty$.
\end{Theorem} 	
\begin{proof}
First of all, notice that $\bar{z}$ belongs 
to the resolvent set of the adjoint~$H^*$ 
and $R(z)^* = (H^*-\bar{z} I)^{-1}$ 
is a compact operator on~$\Hilbert$, too. 
Indeed, since the residual spectrum of~$H$ is empty
(\cf~Proposition~\ref{Prop.residual}),
$\bar{z} \not\in \sigma_\mathrm{p}(H^*)$
and the operator $H^*-\bar{z} I$ is invertible.
Since~$H^*$ is also complex-self-adjoint,
its residual spectrum is empty, too,
so $\Ran(H^*-\bar{z} I)$ is dense in~$\Hilbert$.
Since $\Ran(H^*-\bar{z} I)$ is closed if, and only if,
$\Ran(H-z I)$ is closed (\cf~\cite[Thm.~I.3.7]{Edmunds-Evans})
and~$z$ belongs to the resolvent set of~$H$,
it follows that actually $\Ran(H^*-\bar{z} I) = \Hilbert$.
Hence, $R(z)^*$ is defined on~$\Hilbert$.
This fact enables us to conclude that 
$R(z)$ is $C$-self-adjoint:
$$
  CR(z)C^{-1} 
  = C (H-z I)^{-1}
  \underbrace{C^{-1} (H^*-\bar{z}I) C}_{H-z I}
  C^{-1} (H^*-\bar{z}I)^{-1} 
  = R(z)^*	 
  \,.
$$

Now, by Theorem~\ref{Thm.compact} applied to~$R(z)$,
there exists an orthonormal basis 
$(\phi_j)_{j=1}^\infty$ of~$\Hilbert$ 
(composed of eigenvectors of $|R(z)|$)
such that 
\begin{equation}\label{order}
  R(z) \phi_j = \sigma_j C^{-1} \phi_j,
\end{equation}
where~$\sigma_j$ are the singular values of $R(z)$.
Recall that $(\sigma_j)_{j=1}^\infty$ is a decreasing sequence 
of positive numbers such that $\sigma_j \to 0$ as $j \to \infty$. 
Since each~$C^{-1}\phi_j$ belongs to $\Ran(R(z)) = \Dom(H)$, 
we apply $H-z I$ to both sides of~\eqref{order}
and the desired result~\eqref{expansion} 
follows with $\lambda_j := \sigma_j^{-1}$
and $\psi_j := C^{-1} \phi_j$.	 
\end{proof} 	
\begin{Corollary}
Under the hypotheses of Theorem~\ref{Thm.expansion},
$$
  \|R(z)\| = \frac{1}{\lambda_1}
  \,.
$$
\end{Corollary}

This corollary is a useful tool 
for the study of pseudospectral properties 
of complex-self-adjoint operators.	
Recall that, given any positive number~$\eps$, 
the \emph{pseudospectrum} $\sigma_\eps(H)$ of a closed linear operator~$H$
is the union of its spectrum $\sigma(H)$
and all the complex numbers~$z$  	
such that $\|R(z)\| > \eps^{-1}$.

\begin{Corollary}
Under the hypotheses of Theorem~\ref{Thm.expansion},
if $\lambda_1 > \eps$, then $z \in \sigma_\eps(H)$.
\end{Corollary}
%

\section{Toy model: spinorial Hilbert space}\label{Sec.toy}
%
In this section we introduce a one-parameric family 
of generically non-self-adjoint operators 
which are complex-self-adjoint with respect to an antiunitary operator,
but the latter cannot be chosen,
in any obvious way, to be involutive. 
Our motivation is the anti-involutive
time-reversal operator for fermionic systems.

We begin with the self-adjoint one-dimensional Pauli operator 	
in the Hilbert space $\Hilbert := \sii(\Real;\Com^2)$ defined by
$$
  H_0 := 
  \begin{pmatrix}
    p^2 & 0 \\
    0 & p^2
  \end{pmatrix}
  \,,
  \qquad
  \Dom(H_0) := W^{2,2}(\Real;\Com^2)
  \,,
$$
where $p\varphi:=-i\varphi'$ with $\Dom(p):=W^{1,2}(\Real)$
is the (self-adjoint) \emph{momentum operator} in $\sii(\Real)$. 	
It is well known that~$H_0$ is self-adjoint 
and $\sigma(H_0) = [0,+\infty)$. 	
Physically, $H_0$~represents 
the quantum Hamiltonian of 
a non-relativistic spin $\frac{1}{2}$ particle on a line.

Obviously, $H_0$~is complex-self-adjoint 
with respect to a variety of antiunitary operators on~$\Hilbert$, 
for instance,
$$
  K\psi := \bar{\psi} 
  \,, \qquad
  C_1 := \sigma_1 K
  \,, \qquad
  C_2 := -i\sigma_2 K
  \,, \qquad
  C_3 := \sigma_3 K
  \,,
$$  	
where
$$
  \sigma_1 :=
  \begin{pmatrix}
    0 & 1 \\
    1 & 0
  \end{pmatrix}
  \,, \qquad
  \sigma_2 :=
  \begin{pmatrix}
    0 & -i \\
    i & 0
  \end{pmatrix}
  \,, \qquad
  \sigma_3 :=
  \begin{pmatrix}
    1 & 0 \\
    0 & -1
  \end{pmatrix}
  \,,
$$
are the \emph{Pauli matrices}. 	
Here $K,C_1,C_3$ are involutive, while $C_2$ is anti-involutive. 	
For the Schr\"odinger equation generated 
by the scalar Hamiltonian~$p^2$ in $\sii(\Real)$,
the time-reversal operator is represented 
by the complex conjugation~$K$.
For the Pauli equation, however, 
it is rather~$C_2$ (or its multiple)
which plays the role of the time-reversal operator
(see, \eg, \cite[Sec.~3]{Sachs}). 

Now, let us consider the off-diagonal perturbation
$$
  V_\alpha := 
  \begin{pmatrix}
    0 & p \\
    \alpha p & 0
  \end{pmatrix}
  \,,
  \qquad
  \Dom(V) := W^{1,2}(\Real;\Com^2)
  \,,
$$	
where~$\alpha$ is a real parameter.  	
For every $\alpha \in \Real$, it is easy to see that
$V_\alpha$ is relatively bounded with respect to~$H_0$
with the relative bound less than one
(actually, the bound can be chosen arbitrarily small). 	
Consequently, $H_\alpha := H_0+V_\alpha$ 
with its natural domain 
$
  \Dom(H_\alpha) = \Dom(H_0) \cap \Dom(V_\alpha)
  = W^{2,2}(\Real;\Com^2)
$	
is closed. 	
$H_\alpha$ is not self-adjoint unless $\alpha=1$.
Indeed,
$$
  H_\alpha^* = 
  \begin{pmatrix}
    p^2 & \alpha p \\
    p & p^2
  \end{pmatrix}
  \,,
  \qquad
  \Dom(H_\alpha^*) = \Dom(H_\alpha)
  = W^{2,2}(\Real;\Com^2)
  \,.
$$ 	
 	
It is straightforward to check that~$H_\alpha$ 
is not $C_1$-self-adjoint,
irrespectively of the value of~$\alpha$. 
At the same time, $H_\alpha$ is not complex-self-adjoint 
with respect to~$K$ (respectively, $C_3$),
unless $\alpha=-1$ (respectively, $\alpha=1$).	
However, $H_\alpha$ is $C_2$-self-adjoint,
for every $\alpha \in \Real$.
What is more, there is no \emph{obvious} conjugation with respect 
to which~$H_\alpha$ is complex-self-adjoint, unless $\alpha= \pm 1$.	

\begin{Proposition}
Let $\alpha \in \Real\setminus\{\pm 1\}$.
There exists no antiunitary operator of multiplication
by a constant matrix~$C$
satisfying $C^2=I$ and $H_\alpha^* = C H_\alpha C^{-1}$. 
\end{Proposition}	
\begin{proof}
By contradiction, let us assume that there exists 
an antiunitary operator of multiplication
by a constant matrix~$C$ 
satisfying $C^2=I$ and $H_\alpha^* = C H_\alpha C^{-1}$. 
Writing $C = AK$, where~$K$ is the complex conjugation,
it follows that~$A$ is necessarily a unitary matrix of multiplication, 
say
$$
  A = 
  \begin{pmatrix}
    a_{11} & a_{12} \\
    a_{21} & a_{22}
  \end{pmatrix}
  \,, \qquad
  a_{11}, a_{12}, a_{21}, a_{22} \in \Com
  \,.
$$   
Hence, it is enough to show that the equations
\begin{equation}\label{inconsistency}
  H_\alpha^* A = A KH_\alpha K
  \,, \qquad
  A\bar{A}^T = I
  \qquad \mbox{and} \qquad
  A\bar{A} = I
\end{equation}
cannot be satisfied simultaneously.
From the last two equations of~\eqref{inconsistency},
we particularly deduce that $a_{12}=a_{21}$.
Then, using $\bar{p}=-p$, 
the first equation of~\eqref{inconsistency} reads
$$
  \begin{pmatrix}
     a_{11} p^2+a_{12}\alpha p 
     & a_{12} p^2+a_{22}\alpha p \\
     a_{11} p+a_{12}p^2 
     & a_{12} p+a_{22} p^2 \\
  \end{pmatrix}
  =
  \begin{pmatrix}
     a_{11} p^2-a_{12}\alpha p 
     & -a_{11} p+a_{12} p^2 \\
     a_{12} p^2 - a_{22} \alpha p 
     & -a_{12} p + a_{22} p^2 \\
  \end{pmatrix}
  \,.
$$
From the equations on the diagonal we deduce $a_{12} = 0$,
while the off-diagonal equations yield 
$a_{22} \alpha = - a_{11}$ and $a_{11}=-\alpha a_{22}$.
Taking the absolute value and using that $|a_{11}|=1=|a_{22}|$, 
it follows that $|\alpha|=1$.
\end{proof}

We leave as an open problem whether~$H_\alpha$
admits a more refined \emph{involutive} antiunitary operator
with respect to which it is complex-self-adjoint.  	
 	
Using the Fourier transform, the spectrum of~$H_\alpha$
can be computed explicitly:
$$
  \sigma(H_\alpha) =
\begin{cases}
  \displaystyle
  \left[ -\frac{\alpha}{4},+\infty \right)
  & \mbox{if} \quad \alpha\geq 0 \,,
  \\
  \left\{ \lambda \in \Com : \
  \Re\lambda \geq 0 \ \land \
  |\Im\lambda|^2 = |\alpha| \, \Re\lambda
  \right\}
  & \mbox{if} \quad \alpha < 0 \,.
\end{cases}
$$	
It is interesting that the spectrum is purely real 
for all non-negative~$\alpha$ despite the fact that
the operator~$H_\alpha$ is self-adjoint only if $\alpha=1$. 
Since the reality of the spectrum is usually associated 
with some symmetries, let us mention that~$H_\alpha$
is ``$PC_2$-symmetric'', \ie\ $[H_\alpha,PC_2]=0$, with $P:=\sigma_1$. 
When $\alpha=0$, there is an abrupt transition
to a parabolic curve which converges to the imaginary axis
as $\alpha \to -\infty$. 	
 	
Pseudospectral properties of~$H_\alpha$ 
as well as perturbations by matrix-valued potentials 
should constitute an interesting area of a future research.

\section{Further examples}\label{Sec.examples}
%
The objective of this last section is to collect other examples
of complex-self-adjoint operators with no obvious involutive conjugations,
in a different Hilbert-space setting.
We refer to \cite{Garcia-Mashreghi-Ross} 
for more details on the present functional setup.

In what follows, $\sii(\Circle) =: \sii$ denotes the Lebesgue space
of all complex-valued, square-integrable functions 
on the unit circle~$\Circle$.
The Hardy space $H^2(\Disk) =: H_+^2$ of the unit disk~$\Disk$ 
is defined by
$$
  H^2(\Disk) := \{f \in \sii : \ 
  \hat{f}(n) = 0 \ \mbox{ for } \ n<0 \}
  \,,
$$ 
where $\hat{f}(n)$ is the $n^\mathrm{th}$ Fourier coefficient of $f \in \sii$.
We also introduce
$$
  H_-^2 := \bar{z} \overline{H_+^2} 
  = \{f \in \sii : \ 
  \hat{f}(n) = 0 \ \mbox{ for } \ n \geq 0 \}
  \,.
$$
We have $\sii = H_+^2 \oplus H_-^2$.
The orthogonal projections from~$\sii$ to $H_\pm^2$
will be denoted by $P^\pm$.
By~$L^\infty$ we denote the space of all complex-valued,
essentially bounded functions on~$\Circle$.
By~$H^\infty$ we denote the space of all functions
analytic and bounded on~$\Disk$,
identified with a closed subspace of~$L^\infty$.

Given an inner function 
(\ie\ $\gamma \in H^\infty$ and $|\gamma|=1$ a.e.\ on~$\Circle$),
we associate to it the \emph{model space}
$$
  K_\gamma := H^2 \ominus \gamma H^2
  = \{f_+ \in H_+^2 : \
  \bar{\gamma} f_+ = f_- \in H_-^2 \} 
  \,.
$$ 
If~$\gamma_1$ is another inner function such that 
$\gamma / \gamma_1 \in H^\infty$,
we say that~$\gamma_1$ \emph{divides}~$\gamma$
(we write $\gamma_1 \preceq \gamma$)
and we can decompose 
$$
  K_\gamma = K_{\gamma_1} \oplus \gamma_1 K_{\gamma/\gamma_1}
  \,.
$$ 

In each model space~$K_\gamma$ we can define a natural
involutive conjugation~$C_\gamma$ given by
\begin{equation}\label{natural}
  C_\gamma f := \gamma \bar{z} \bar{f}
\end{equation}
for every $f \in K_\gamma$.

\subsection{Example 1}\label{Ex1}
Let $\gamma,\alpha,\beta$ be inner functions, with $\gamma = \alpha\beta$.
Then $K_\gamma$ admits two orthogonal decompositions,
$$
  K_\gamma = K_\alpha \oplus \alpha K_\beta
  \,, \qquad
  K_\gamma = K_\beta \oplus \beta K_\alpha
  \,.
$$
Consequently, every function $f \in K_\gamma$ admits two representations
$$
  f = f_{1,\alpha} + \alpha f_{2,\beta}
  \,, \qquad
  f = f_{1,\beta} + \beta f_{2,\alpha}
  \,,
$$
with $f_{1,\alpha}, f_{2,\alpha} \in K_\alpha$
and $f_{1,\beta}, f_{2,\beta} \in K_\beta$.

Let $\xi \in \Real$ and define the antilinear operator $C_{\alpha,\beta}$
on~$K_\gamma$ by setting, for every $ f \in K_\gamma$,
$$
  C_{\alpha,\beta}(f) 
  = C_{\alpha,\beta}(f_{1,\alpha} + \alpha f_{2,\beta})
  := e^{i\xi} C_\beta f_{2,\beta} + \beta C_\alpha f_{1,\alpha} 
  \,,
$$
where $C_\beta(\varphi) = \beta \bar{z} \bar\varphi$
for $\varphi \in K_\beta$ 
and $C_\alpha(\varphi) = \alpha \bar{z} \bar\varphi$
for $\varphi \in K_\alpha$, 
according to~\eqref{natural}.  

\begin{Proposition}
One has, for every $ f \in K_\gamma$,
$$
  C_{\alpha,\beta}^*(f) 
  = C_{\alpha,\beta}^*(f_{1,\beta} + \beta f_{2,\alpha})
  = C_\alpha f_{2,\alpha} + e^{i\xi} \alpha C_\beta f_{1,\beta} 
  \,.
$$
\end{Proposition}
\begin{proof}
Let us abbreviate $C_{\alpha,\beta} =: C$.
For any $f,g \in K_\gamma$ with $f = f_{1,\alpha} + \alpha f_{2,\beta}$
and $g = g_{1,\beta} + \beta g_{2,\alpha}$, one has
$$
  (g,Cf) 
  = \big(g_{1,\beta} + \beta g_{2,\alpha},
  C(f_{1,\alpha} + \alpha f_{2,\beta})\big)
  = \big(g_{1,\beta} + \beta g_{2,\alpha}, 
  e^{i\xi} C_\beta f_{2,\beta} + \beta C_\alpha f_{1,\alpha} \big)
  \,.
$$ 
Since $e^{i\xi} C_\beta f_{2,\beta} \in K_\beta \bot \beta H_+^2$
and $\beta C_\alpha f_{1,\alpha} \in \beta H_+^2 \bot K_\beta$,
we get
$$
\begin{aligned}
  (g,Cf)
  &= (g_{1,\beta},e^{i\xi} C_\beta f_{2,\beta})
  + (\beta g_{2,\alpha},\beta C_\alpha f_{1,\alpha})
  \\
  &= (e^{-i\xi} g_{1,\beta}, C_\beta f_{2,\beta})
  + (g_{2,\alpha},C_\alpha f_{1,\alpha})
  \\
  &= \big(f_{2,\beta},C_\beta(e^{-i\xi} g_{1,\beta})\big)
  + (f_{1,\alpha},C_\alpha g_{2,\alpha})
  \\
  &= (f_{2,\beta},e^{i\xi}C_\beta g_{1,\beta})
  + (f_{1,\alpha},C_\alpha g_{2,\alpha})
  \\
  &= (\alpha f_{2,\beta},\alpha e^{i\xi}C_\beta g_{1,\beta})
  + (f_{1,\alpha},C_\alpha g_{2,\alpha})
  \\
  &= (f_{1,\alpha} + \alpha f_{2,\beta},
  C_\alpha g_{2,\alpha} + \alpha e^{i\xi}C_\beta g_{1,\beta}) 
  \\
  &= (f,C^* g)
  \,,
\end{aligned}  
$$
as required by the adjoint of an antilinear operator.
\end{proof}
\begin{Proposition}
$C_{\alpha,\beta}$ is antiunitary. 
\end{Proposition}
\begin{proof}
For every $f \in K_\gamma$, one has 
(again abbreviating $C_{\alpha,\beta} =: C$)
$$
\begin{aligned}
  C^* C f = C^* C (f_{1,\alpha} + \alpha f_{2,\beta})
  &= C^* (e^{i\xi} C_\beta f_{2,\beta} + \beta C_\alpha f_{1,\alpha})
  \\
  &= C_\alpha(C_\alpha f_{1,\alpha})
  + e^{i\xi} \alpha C_\beta(e^{i\xi} C_\beta f_{2,\beta} )
  \\
  &= f_{1,\alpha} + \alpha f_{2,\beta} = f \,.
\end{aligned}  
$$
Analogously, one can show $C C^* f = f$.
\end{proof}
\begin{Remark}
If $\alpha=\beta$, we have
$
  C_{\alpha,\alpha} (f_{1,\alpha} + \alpha f_{2,\alpha})
  = e^{i\xi} f_{2,\alpha} + \alpha f_{1,\alpha}
$
and 
$$
  C_{\alpha,\alpha}^2 (f_{1,\alpha} + \alpha f_{2,\alpha})
  = e^{i\xi} f_{1,\alpha} + \alpha e^{-i\xi} f_{2,\alpha}
  \,.
$$
If $e^{i\xi} =1$, then $C_{\alpha,\alpha}^2 = I$
and it turns out that $C_{\alpha,\alpha}$ is the usual (involutive) 
conjugation $C_{\alpha^2}$ on $K_{\alpha^2}$, see~\eqref{natural}.
On the other hand,
if $e^{i\xi} =-1$, then  
we have $C_{\alpha,\alpha}^2 = -I$, 
with 
$
  C_{\alpha,\alpha} f  
  = -C_\alpha f_{2,\alpha} + \alpha C_\alpha f_{1,\alpha}
$.
\end{Remark}

As an example of an operator on a model space $K_{\alpha^2}$
which is $C_{\alpha,\alpha}$-self-adjoint,
but not $C_{\alpha^2}$-self-adjoint,
consider the case $e^{i\xi} =-1$,
$\alpha := z^2$ and the operator~$H$ defined on $K_{z^4}$
(the space of polynomials of degree not greater than~$3$) by 
(with respect to the standard basis of $K_{z^4}$)
$$
  H := 
  \begin{pmatrix}
    a_{11} & a_{12} & a_{13} & 0 \\
    a_{21} & a_{22} & 0 & -a_{13} \\
    a_{31} & 0 & a_{22} & a_{12} \\
    0 & - a_{31} & a_{21} & a_{11} 
  \end{pmatrix}
  ,
$$
where $a_{jk}$ with $j,k \in \{1,2,3\}$ are complex numbers.
Then (abbreviating $C_{\alpha,\alpha} =: C$)
$$
  CHC^{-1} = - CHC = H^* \,.
$$

\subsection{Example 2}
Let~$C$ be the antilinear operator defined on~$\sii$ by
$$
  Cf := \frac{z-\bar{z}}{2} \ \overline{f(\bar{z})}
  + \frac{z+\bar{z}}{2} \ \overline{f(-\bar{z})}
  = z \ \frac{\overline{f(\bar{z})}+\overline{f(-\bar{z})}}{2}
  - \bar{z} \ \frac{\overline{f(\bar{z})}-\overline{f(-\bar{z})}}{2}
  \,.
$$

\begin{Proposition}
$C^* = -C$.
\end{Proposition}
\begin{proof}
For every $f,g \in \sii$, one has
$$
  (g,Cf) = \left(
  \frac{g(z)+g(-z)}{2} + \frac{g(z)-g(-z)}{2},
  z \ \frac{\overline{f(\bar{z})}+\overline{f(-\bar{z})}}{2}
  - \bar{z} \ \frac{\overline{f(\bar{z})}-\overline{f(-\bar{z})}}{2}
  \right)
  .
$$
Noting that
$$
  \frac{g(z)+g(-z)}{2} 
  \ \mbox{\Large $\bot$} \ 
  z \ \frac{\overline{f(\bar{z})}+\overline{f(-\bar{z})}}{2}
  \qquad\mbox{and}\qquad
  \frac{g(z)-g(-z)}{2} 
  \ \mbox{\Large $\bot$} \ 
  \bar{z} \ \frac{\overline{f(\bar{z})}-\overline{f(-\bar{z})}}{2}
  \,,
$$ 
we get
$$
\begin{aligned}
  (g,Cf) &=
  \left(
  \frac{g(z)-g(-z)}{2} , 
  z \ \frac{\overline{f(\bar{z})}+\overline{f(-\bar{z})}}{2}
  \right)
  - \left( 
  \frac{g(z)+g(-z)}{2} ,
  \bar{z} \ \frac{\overline{f(\bar{z})}-\overline{f(-\bar{z})}}{2}
  \right)
  \\
  &= 
  \left(
  \bar{z} \ \frac{g(z)-g(-z)}{2} , 
  \frac{\overline{f(\bar{z})}+\overline{f(-\bar{z})}}{2}
  \right)
  - \left( 
  z \ \frac{g(z)+g(-z)}{2} ,
  \frac{\overline{f(\bar{z})}-\overline{f(-\bar{z})}}{2}
  \right)
  \\
  &= 
  \left(
  \frac{f(\bar{z})+f(-\bar{z})}{2} ,
  z \ \frac{\overline{g(z)}-\overline{g(-z)}}{2}
  \right)
  - \left( 
  \frac{f(\bar{z})-f(-\bar{z})}{2} ,
  \bar{z} \ \frac{\overline{g(z)}+\overline{g(-z)}}{2} 
  \right)
  \\
  &= 
  \left(
  \frac{f(z)+f(-z)}{2} ,
  \bar{z} \ \frac{\overline{g(\bar{z})}-\overline{g(-\bar{z})}}{2}
  \right)
  + \left( 
  \frac{f(z)-f(-z)}{2} ,
  -z \ \frac{\overline{g(\bar{z})}+\overline{g(-\bar{z})}}{2} 
  \right)
  \\
  &= 
  \left(
  f(z) ,
  \bar{z} \ \frac{\overline{g(\bar{z})}-\overline{g(-\bar{z})}}{2}
  - z \ \frac{\overline{g(\bar{z})}+\overline{g(-\bar{z})}}{2} 
  \right)
  = (f,-Cg) \,,
\end{aligned}  
$$
so the desired claim follows by the definition 
of the adjoint of an antilinear operator.
\end{proof}

It is straightforward to verify the following claims.
\begin{Proposition}
$C$ is antiunitary and $C^2 = -I$.
\end{Proposition}

It is perhaps less obvious that the Hardy space $H_+^2$
is left invariant by~$C$.
\begin{Proposition}\label{Prop3}
$C(H_+^2)=H_+^2$.  
\end{Proposition}
\begin{proof}
The claim follows from the observation that for $f \in H_+^2$
with $f(z) = a_0 + a_1 z + a_2 z^2 + a_3 z^3 + \dots$
we have that 
$Cf(z) = -\bar{a}_1 + \bar{a}_0 z - \bar{a}_3 z^2 + \bar{a}_2 z^3 + \dots$
\end{proof}
\begin{Remark}
The usual conjugation in $H_+^2$ is given by
(\cf~\cite{Ko-Lee_2016,Camara-Klis-Garlicka-Lanucha-Ptak_2020})
$$
  \tilde{C}f (z) := \overline{f(\bar{z})}
  \,,
$$ 
so that for $f(z) = a_0 + a_1 z + a_2 z^2 + \dots$
we have that $\tilde{C}f(z) = \bar{a}_0 + \bar{a}_1 z + \bar{a}_2 z^2 + \dots$
\end{Remark} 
\begin{Corollary}
$CP^\pm = P^\pm C$.
\end{Corollary} 

From Proposition~\ref{Prop3},
the restriction $C |_{H_+^2}$
is an anti-involutive conjugation on $H_+^2$,
which we will also denote by~$C$.

\begin{Proposition} 
If~$\theta$ is an inner function such that 
\begin{equation}\label{theta}
  \theta(\bar{z}) = \theta(-\bar{z}) = \overline{\theta(z)}
\end{equation}
for all $z \in \Disk$, 
then $C(K_\theta) = K_\theta$ and 
$\tilde{C}_\theta := C |_{K_\theta}$
defines an anti-involutive conjugation on~$K_\theta$.
\end{Proposition}
\begin{proof}
We have that $f_+ \in K_\theta$ if, and only if,
$f_+ \in H_+^2$ and $\bar\theta f_+ = f_-$ 
where $f_- \in H_-^2$.  
Let~$f_+$ be any element of~$K_\theta$.
Then 
$$
\begin{aligned}
  \overline{\theta(z)} f_+(z) = f_-(z)
  & \Longleftrightarrow 
  \theta(\bar{z}) \overline{f_+(\bar{z})} = \overline{f_-(\bar{z})}
  \\
  & \Longleftrightarrow 
  \overline{\theta(z)} \, \overline{f_+(\bar{z})} = \overline{f_-(\bar{z})}
  \\
  & \Longleftrightarrow 
  \overline{\theta(-z)} \, \overline{f_+(-\bar{z})} = \overline{f_-(-\bar{z})}
  \\
  & \Longleftrightarrow 
  \overline{\theta(z)} \, \overline{f_+(-\bar{z})} = \overline{f_-(-\bar{z})}
  \,.
\end{aligned} 
$$
Consequently,
$$
\begin{aligned}
  \bar\theta \tilde{C}_\theta f(z) 
  &= \overline{\theta(z)} 
  \left(
  z \ \frac{\overline{f_+(\bar{z})}+\overline{f_+(-\bar{z})}}{2}
  - \bar{z} \ \frac{\overline{f_+(\bar{z})}-\overline{f_+(-\bar{z})}}{2}
  \right)
  \\
  &= z \ \frac{\overline{f_-(\bar{z})}+\overline{f_-(-\bar{z})}}{2}
  - \bar{z} \ \frac{\overline{f_-(\bar{z})}-\overline{f_-(-\bar{z})}}{2}
  \in H_-^2
  \,.
\end{aligned} 
$$
Therefore $C(K_\theta) \subset K_\theta$
and since $C^2 = -I$ we have also $K_\theta \subset C(K_\theta)$.
\end{proof}
\begin{Remark}
It follows that $C |_{K_\theta}$ is also an anti-involutive 
conjugation on~$K_\theta$ if~$\theta$ satisfies~\eqref{theta}.   
\end{Remark} 

For example, if $\theta := z^4$ and
$
  f_+(z) = a_0 + a_1 z + a_2 z^2 + a_3 z^3 \in K_{z^4}
$, 
where $a_0,a_1,a_2,a_3$ are complex numbers,
then
$
  Cf_+(z) 
  = -\bar{a}_1 + \bar{a}_0 z - \bar{a}_3 z^2 + \bar{a}_2 z^3 \in K_{z^4}
$.
(Comparing with the conjugation $C_{z^2,z^2}$ on $K_{z^4}$,
defined in Section~\ref{Ex1}, we have 
$
  C_{z^2,z^2}f_+(z) 
  = -\bar{a}_3 - \bar{a}_2 z + \bar{a}_1 z^2 + \bar{a}_0 z^3 
$.) 
An operator~$H$ defined on $K_{z^4}$ is $C$-self-adjoint
if it is defined by a complex matrix of the form
$$
  H := 
  \begin{pmatrix}
    a_{11} & 0 & a_{13} & a_{14} \\
    0 & a_{11} & a_{23} & a_{24} \\
    a_{24} & -a_{14} & a_{33} & 0 \\
    - a_{23} & a_{13} & 0 & a_{33} 
  \end{pmatrix}
  .
$$

As an example of an operator on $H_+^2$
which is $C$-self-adjoint (denoting $C |_{H_+^2} =: C$),
one can consider the following natural generalisation of Toeplitz operators 
(defined, for $\varphi \in L^\infty$,
as $T_\varphi: H_+^2 \to H_+^2$, 
$T_\varphi f_+ := P^+ \varphi f_+$),
with $\varphi_1, \varphi_2 \in L^\infty$:
$$
  T_{\varphi_1, \varphi_2} : H_+^2 \to H_+^2
  \,, \qquad
  \begin{aligned}
    T_{\varphi_1, \varphi_2} (f_+) := \ &
    P^+ \left(
    \varphi_1(z) \ \frac{f_+(z)+f_+(-z)}{2} 
    +  \varphi_2(z) \ \frac{f_+(z)-f_+(-z)}{2} 
    \right)
    \\
    = \ &
    P^+ \left(
    \frac{\varphi_1(z)+\varphi_2(z)}{2}  \ f_+(z)
    + \frac{\varphi_1(z)-\varphi_2(z)}{2}  \ f_+(-z)
    \right) .
  \end{aligned}
$$
We shall occasionally abbreviate $T_{\varphi_1, \varphi_2} =: T$.

\begin{Proposition}\label{Prop.8A} 
$T_{\varphi_1, \varphi_2} = 0$, if and only if,
$\varphi_1= \varphi_2 = 0$.
\end{Proposition}
\begin{proof}
We have $T_{\varphi_1, \varphi_2} = 0$, if and only if,
$\Ker(T_{\varphi_1, \varphi_2}) = H_+^2$,
\ie, for any $f_+ \in H_+^2$
$$
  \varphi_1(z) \ \frac{f_+(z)+f_+(-z)}{2} 
  + \varphi_2(z) \ \frac{f_+(z)-f_+(-z)}{2} 
  \in H_-^2
  \,.
$$
In particular, this must hold for any even function $f_+ \in H_+^2$.
Taking $f_+ = 1$, this implies that $\varphi_1 \in H_-^2$;
taking $f_+ = z^{2n}$ with $n \in \Nat$, we conclude that
$\varphi_1 \in (\bar{z}^2)^n H_-^2$ for all $n \in \Nat$
and, therefore, $\varphi_1 = 0$.
Analogously, taking~$f_+$ to be any odd function in~$H_+^2$,
we conclude that $\varphi_2 = 0$.
\end{proof}
\begin{Corollary}\label{Corol.8B} 
For $\varphi_1, \varphi_2, \tilde\varphi_1, \tilde\varphi_2 \in L^\infty$,
we have  
$$
  T_{\varphi_1, \varphi_2} = T_{\tilde\varphi_1, \tilde\varphi_2} 
  \quad \Longleftrightarrow \quad
  (\varphi_1 = \tilde\varphi_1
  \ \land \
  \varphi_2 = \tilde\varphi_2)
  \,.
$$
\end{Corollary}
\begin{Proposition} 
We have, for every $f_+ \in H_+^2$,
$$ 
  T_{\varphi_1, \varphi_2}^* (f_+) 
  = P^+ \left(
    \frac{\overline{\varphi_1(z)}+\overline{\varphi_2(z)}}{2}  \ f_+(z)
    + \frac{\overline{\varphi_1(-z)}-\overline{\varphi_2(-z)}}{2}  \ f_+(-z)
    \right) .
$$  
\end{Proposition}
\begin{proof}
For any $f_+,g_+ \in H_+^2$, we have
$$
\begin{aligned}
  (g_+,T f_+) &=
  \left( 
  g_+(z), 
  P^+ \left(
    \frac{\varphi_1(z)+\varphi_2(z)}{2}  \ f_+(z)
    + \frac{\varphi_1(z)-\varphi_2(z)}{2}  \ f_+(-z)
  \right)     
  \right)
  \\
  &= \left( 
  g_+(z),  
    \frac{\varphi_1(z)+\varphi_2(z)}{2}  \ f_+(z)
  \right)
  +  \left( 
  g_+(z),  
    \frac{\varphi_1(z)-\varphi_2(z)}{2}  \ f_+(-z)
  \right)
  \\
  &= \left( 
  \frac{\overline{\varphi_1(z)}+\overline{\varphi_2(z)}}{2} \ g_+(z),  
  f_+(z)
  \right)
  +  \left( 
  \frac{\overline{\varphi_1(z)}-\overline{\varphi_2(z)}}{2} \ g_+(z),  
  f_+(-z)
  \right)
  \\
  &= \left( 
  \frac{\overline{\varphi_1(z)}+\overline{\varphi_2(z)}}{2} \ g_+(z),  
  f_+(z)
  \right)
  +  \left( 
  \frac{\overline{\varphi_1(-z)}-\overline{\varphi_2(-z)}}{2} \ g_+(-z),  
  f_+(z)
  \right)
  \\
  &=
  \left( 
  P^+ \ \frac{\overline{\varphi_1(z)}+\overline{\varphi_2(z)}}{2} \ g_+(z),
  f_+(z)
  \right)
  + \left( 
  P^+ \ \frac{\overline{\varphi_1(-z)}-\overline{\varphi_2(-z)}}{2} \ g_+(-z),
  f_+(z)
  \right)
  \,,
\end{aligned}
$$
which establishes the desired claim.
\end{proof}
\begin{Proposition} 
We have $C T_{\varphi_1, \varphi_2} C^{-1} = T_{\varphi_1, \varphi_2}^*$ 
if, and only if, 
\begin{equation}\label{and}
  -\bar{z} \varphi_1(z) + z \varphi_2(z) 
  = z \varphi_1(\bar{z}) - \bar{z} \varphi_2(\bar{z}) 
  \\
  \quad\mbox{and}\quad
  \bar{z} \varphi_1(z) + z \varphi_2(z) 
  = z \varphi_1(-\bar{z}) + \bar{z} \varphi_2(-\bar{z}) \,.
\end{equation}
\end{Proposition}
\begin{proof}
On the one hand, we have
$$
\begin{aligned}
  T C f_+ &= P^+ \Bigg[
  \frac{\varphi_1(z)+\varphi_2(z)}{2}
  \left(
  \frac{z-\bar{z}}{2} \, \overline{f_+(\bar{z})}
  + \frac{z+\bar{z}}{2} \, \overline{f_+(-\bar{z})}
  \right)
  \\
  & \quad\qquad
  +\frac{\varphi_1(z)-\varphi_2(z)}{2}
  \left(
  -\frac{z-\bar{z}}{2} \, \overline{f_+(-\bar{z})}
  - \frac{z+\bar{z}}{2} \, \overline{f_+(-\bar{z})}
  \right)
  \Bigg]
  \\
  &= 
  \frac{1}{2} \, P^+ \left[
  (-\bar{z}\varphi_1(z) + z \varphi_2(z)) 
  \, \overline{f_+(\bar{z})}
  + (\bar{z}\varphi_1(z) + z \varphi_2(z)) 
  \, \overline{f_+(-\bar{z})}
  \right] .
\end{aligned}  
$$
On the other hand, we have
$$
\begin{aligned}
  C T^* f_+ 
  &= C P^+ \left[
  \frac{\overline{\varphi_1(z)}+\overline{\varphi_2(z)}}{2} f_+(z)
  + \frac{\overline{\varphi_1(-z)}-\overline{\varphi_2(-z)}}{2} f_+(-z)
  \right]
  \\
  &= P^+ C \left[
  \frac{\overline{\varphi_1(z)}+\overline{\varphi_2(z)}}{2} f_+(z)
  + \frac{\overline{\varphi_1(-z)}-\overline{\varphi_2(-z)}}{2} f_+(-z)
  \right]
  \\
  &= P_+ \Bigg[
  \frac{z-\bar{z}}{2} 
  \left(
  \frac{\varphi_1(\bar{z})+\varphi_2(\bar{z})}{2} \
  \overline{f_+(\bar{z})}
  + \frac{\varphi_1(-\bar{z})-\varphi_2(-\bar{z})}{2} \
  \overline{f_+(-\bar{z})}
  \right) 
  \\
  & \quad\qquad
  +\frac{z+\bar{z}}{2} 
  \left(
  \frac{\varphi_1(-\bar{z})+\varphi_2(-\bar{z})}{2} \
  \overline{f_+(-\bar{z})}
  + \frac{\varphi_1(\bar{z})-\varphi_2(\bar{z})}{2} \
  \overline{f_+(\bar{z})}
  \right) 
  \Bigg]
  \\
  &= 
  \frac{1}{2} \, P^+ \left[
  (z\varphi_1(\bar{z}) - \bar{z} \varphi_2(\bar{z})) 
  \, \overline{f_+(\bar{z})}
  + (z\varphi_1(-\bar{z}) + \bar{z} \varphi_2(-\bar{z})) 
  \, \overline{f_+(-\bar{z})}
  \right] .
\end{aligned}  
$$
Consequently, 
taking Corollary~\ref{Corol.8B} into account,
we have $TC = CT^*$ if, and only if,~\eqref{and} holds. 
\end{proof}
\begin{Remark}
Condition~\eqref{and} is particularly satisfied if
$$
  \varphi_1(\bar{z}) = \varphi_2(z) 
  \qquad \mbox{and} \qquad
  \frac{\varphi_1(z) -\varphi_1(-z)}{z}
  = - \frac{\varphi_1(\bar{z}) -\varphi_1(-\bar{z})}{\bar{z}}
  \,.
$$ 
\end{Remark} 

\subsection{Example 3}
Let $D_j$, with $j \in \{1,2,3,4\}$, be antilinear operators
on a Hilbert space~$\Hilbert$ and consider the antilinear operator
defined on $\Hilbert\oplus\Hilbert$ 
(which we identify with $\Hilbert^2$) by 
$$
  C := 
  \begin{pmatrix}
    D_1 & D_2 \\
    D_3 & D_4
  \end{pmatrix}
  .
$$
We have 
$$
  C^* = 
  \begin{pmatrix}
    D_1^* & D_3^* \\
    D_2^* & D_4^*
  \end{pmatrix}
  .
$$
Consequently,
$$
  CC^* = I 
  \qquad \Longleftrightarrow \qquad
  \left\{
  \begin{aligned}
    D_1 D_1^* + D_2 D_2^* = I \,, 
    \\
    D_3 D_3^* + D_4 D_4^* = I \,,
    \\
    D_1 D_3^* + D_2 D_4^* = 0 \,,
  \end{aligned}
  \right.
$$
and
$$
  C^* C = I 
  \qquad \Longleftrightarrow \qquad
  \left\{
  \begin{aligned}
    D_1^* D_1 + D_3^* D_3 = I \,, 
    \\
    D_2^* D_2 + D_4^* D_4 = I \,,
    \\
    D_1^* D_2 + D_3^* D_4 = 0 \,.
  \end{aligned}
  \right.
$$

Involutive conjugations of this form for $\Hilbert := H_+^2$ 
were studied in \cite{Camara-Klis-Garlicka-Ptak_2019,Ko-Lee-Lee_2022}.
Here we allow for a more abstract setting  
as well as for not necessarily involutive conjugations.
In particular, $C$~will be antinunitary and \emph{anti-involutive}
if, and only if,
$$
  \left\{
  \begin{aligned}
    D_1 &= -D_1^* \,, 
    \\
    D_4 &= -D_4^* \,,
    \\
    D_3 &= -D_2^* \,,
  \end{aligned}
  \right.
  \qquad \mbox{and} \qquad
  \left\{
  \begin{aligned}
    D_2 D_2^* - D_1^2 &= I \,,
    \\
    D_2^* D_2 - D_4^2 &= I \,,
    \\
    D_1 D_2 + D_2 D_4 &= 0 \,. 
  \end{aligned}
  \right.
$$

For example, if~$D$ is antiunitary and anti-involutive, then
$$
  C_1 := \frac{1}{\sqrt{2}}
  \begin{pmatrix}
    -D & D \\
    D & D
  \end{pmatrix}
$$
is also antiunitary and anti-involutive.
Another example is given with $D_1 = 0 = D_4$,
$D_2$ antiunitary and involutive, 
and $D_3 = -D_2^* = -D_2$; then
$$
  C_2 := \frac{1}{\sqrt{2}}
  \begin{pmatrix}
    0 & D_2 \\
    -D_2 & 0
  \end{pmatrix}
$$
is an anti-involutive conjugation on $\Hilbert \oplus \Hilbert$.
\begin{Proposition}\label{Prop11} 
For any operator~$p$ on~$\Hilbert$ such that $D_2 p D_2 = - p^*$,
we have that 
$$
  T_\alpha :=   
  \begin{pmatrix}
    p^2 & \alpha p \\
    p & p^2
  \end{pmatrix}
  , \qquad
  \alpha \in \Real \,,
$$
is $C_2$-self-adjoint.  
\end{Proposition}
\begin{proof}
We have
$$
\begin{aligned}
  C_2T_\alpha C_2^{-1} 
  &=
  \begin{pmatrix}
    0 & D_2 \\
    -D_2 & 0
  \end{pmatrix}
  \begin{pmatrix}
    p^2 & \alpha p \\
    p & p^2
  \end{pmatrix}
  \begin{pmatrix}
    0 & -D_2 \\
    D_2 & 0
  \end{pmatrix}  
  \\
  &=  
  \begin{pmatrix}
    D_2 p^2 D_2 & -D_2 p D_2 \\
    -\alpha D_2 p D_2 & D_2 p^2 D_2
  \end{pmatrix}  
  \\
  &=
  \begin{pmatrix}
    (p^*)^2 & p^* \\
    \alpha p^* & (p^*)^2 
  \end{pmatrix}
  = T_\alpha^*
  \,,
\end{aligned}
$$
which establishes the desired result
\end{proof}
\begin{Remark}
The example of Proposition~\ref{Prop11} 
is of the type of the toy model considered in Section~\ref{Sec.toy},
except that in the former only bounded operators are considered.
\end{Remark} 
%

 	
\subsection*{Acknowledgment}
C.C.\ was partially supported by FCT/Portugal through CAMGSD, IST-ID, projects
UIDB/04459/2020 and UIDP/04459/2020.
D.K.\ was partially supported by the EXPRO grant No.~20-17749X
of the Czech Science Foundation.
D.K. is 
also 
grateful to Instituto Superior T\'ecnico,
where the ideas of this paper were discussed,
for partially supporting his stays in 2020 and 2021.

\newpage
%
%

\providecommand{\bysame}{\leavevmode\hbox to3em{\hrulefill}\thinspace}
\providecommand{\MR}{\relax\ifhmode\unskip\space\fi MR }
\providecommand{\MRhref}[2]{%
  \href{http://www.ams.org/mathscinet-getitem?mr=#1}{#2}
}
\providecommand{\href}[2]{#2}

\end{document}